\documentclass[10pt,a4paper, xcolor]{article}
\usepackage{amsfonts}
\usepackage{amssymb,amsthm, amsmath,graphicx,bm,amsthm}
\usepackage{color,enumerate,lineno}
\usepackage{url}%,showlabels}

\usepackage{amsfonts}
\usepackage{amssymb,amsthm, amsmath,graphicx,bm,amsthm,mathtools}
\usepackage{color,enumerate}
\usepackage{url,lineno}

\usepackage[latin1]{inputenc}

\usepackage{lscape}
\usepackage[all]{xy}
\usepackage{booktabs}

\usepackage{algorithm}
\usepackage[noend]{algpseudocode}

\newtheorem{theorem}{Theorem}

\newtheorem{proposition}[theorem]{Proposition}

\newtheorem{lemma}[theorem]{Lemma}

\theoremstyle{remark}

\newtheorem{example}[theorem]{Example}

\def\gcd{{\mathrm{ gcd }}}
\def\msg{{\mathrm{ msg }}}
\def\Ap{{\mathrm{ Ap }}}

\def\N{\mathbb{N}}

\def\int{\mathrm{int}}

\title{Semigroups with fixed multiplicity and embedding dimension}

\author{
J. I. Garc\'{\i}a-Garc\'{\i}a
\footnote{
	Dpto. de Matem\'aticas/INDESS (Instituto Universitario para el Desarrollo Social Sostenible).
	Universidad de C\'adiz, E-11510 Puerto Real  (C\'{a}diz, Spain).
	E-mail: ignacio.garcia@uca.es.
	Partially supported by MTM2014-55367-P and by Junta de Andaluc\'{\i}a group FQM-366.
}
\and
D. Mar\'{\i}n-Arag\'{o}n
\footnote{
	E-mail: daniel.marinaragon@alum.uca.es.
	Partially supported by Junta de Andaluc\'{\i}a group FQM-366.
}
\and
M. A. Moreno-Fr\'{\i}as
\footnote{
	Dpto. Matem\'aticas,
	Universidad de C\'adiz, E-11510 Puerto Real  (C\'{a}diz, Spain).
	E-mail: mariangeles.moreno@uca.es.
	Partially supported by MTM2014-55367-P and by Junta de Andaluc\'{\i}a group FQM-298.
}
\and
J. C. Rosales
\footnote{
	Dpto. de \'Algebra, Universidad de Granada.
	Partially supported by MTM2014-55367-P and by Junta de Andaluc\'{\i}a group FQM-343.
	E-mail: jrosales@ugr.es.
}
\and
A. Vigneron-Tenorio
\footnote{
	Dpto. de Matem\'aticas/INDESS (Instituto Universitario para el Desarrollo Social Sostenible).
	Universidad de C\'adiz,
 	E-11406 Jerez de la Frontera (C\'{a}diz, Spain).
 	E-mail: alberto.vigneron@uca.es.
 	Partially supported by MTM2015-65764-C3-1-P (MINECO/FEDER, UE) and
	Junta de Andaluc\'{\i}a group FQM-366. 	
}
}

\date{}

\begin{document}

\maketitle

\begin{abstract}
Given $m\in \N,$ a numerical semigroup with multiplicity $m$ is called packed numerical  semigroup if its minimal generating set is included in $\{m,m+1,\ldots, 2m-1\}.$
In this work, packed numerical  semigroups are used to built the set of numerical semigroups with fixed multiplicity and embedding dimension, and to create a partition in this set. Moreover,
Wilf's conjecture is checked in the tree associated to some packed numerical semigroups.


	\smallskip
	{\small \emph{MSC-class:} 20M14 (Primary),  20M05 (Secondary).}
\end{abstract}

\section{Introduction}

Let $\N=\{0,1,2,\ldots\}$ be the set of non-negative integers. A numerical semigroup is a subset $S$ of $\N$ which is closed by sum, $0\in S$ and $\N\backslash S$ is finite. We call multiplicity to the least positive integer in $S$ and we denote it by $m(S)$.

Given a non-empty subset $A$ of $\N$ we denote by $\langle A \rangle$ to the submonoid of $(\N,+)$ generated by $A$, that is, $\langle A \rangle = \{\lambda_1 a_1+\cdots+\lambda_n a_n \mid n\in\N$, $a_1,\ldots,a_n\in A,~\lambda_1,\ldots,\lambda_n\in\N\}$. It is well known (for example, see Lemma 2.1 from \cite{libro}) that $\langle A \rangle$ is a numerical semigroup if and only if $\gcd (A)=1$.

If $S$ is a numerical semigroup and $S=\langle A \rangle,$ we say that $A$ is a system of generators of $S$.
Moreover, $A$ is a minimal system of generators of $S$ if $S\neq\langle B \rangle$ for every $B\subsetneq A$.
In Theorem 2.7 from \cite{libro} it is shown that every numerical semigroup has a unique minimal system of generator and this system is finite. We denote by $\msg(S)$ and $e(S)$ the minimal system of generators of $S$ and its cardinality, also called the embedding dimension of $S$.

In this work, our main aim is to show a procedure that allows us to build recursively the set $\mathcal{L}(m,e)$ formed by all the numerical semigroup with multiplicity $m$ and embedding dimension $e$.

We say that a numerical semigroup $S$ is a packed numerical semigroup if $\msg(S)\subseteq\{m(S),m(S)+1,\ldots,2m(S)-1\}$. The set of all packed numerical semigroups with multiplicity $m$ and embedding dimension $e$ is denoted by $\mathcal{C}(m,e).$

In Section \ref{section_partition}, an equivalence relation $\mathcal{R}$ in the set $\mathcal{L}(m,e)$ is defined. We show that if $S\in\mathcal{L}(m,e)$ then $[S]\cap\mathcal{C}(m,e)$ has cardinality 1, so $\{[S]\mid S\in\mathcal{C}(m,e)\}$ is a partition if $\mathcal{L}(m,e)$.
Hence, for computing all the elements of the set $\mathcal{L}(m,e)$ is only necessary do the following steps:
\begin{enumerate}
    \item Compute $\mathcal{C}(m,e).$
    \item For every $S\in\mathcal{C}(m,e)$ compute $[S]$.
\end{enumerate}

We see that it is easy to compute $\mathcal{C}(m,e)$ because, actually, this problem is equivalent to compute all the subsets $A$ of $\{1,2,\ldots,m-1\}$ such that $A$ has cardinality $e-1$ and $\gcd(A\cup\{m\})=1$. For computing $[S]$ we order its elements making a tree with root in $S,$ and see how the children of the vertices are.

In this way, we can build recursively the elements of $[S]$ adding in each step the children of the vertices we got in the previous step. This procedure is not algorithmic because $[S]$ is infinite and we can not build it in a finite number of steps.

If $S$ is a numerical semigroup, we call Frobenius number (respectively, genus) of $S$ the greater integer which is not in $S$ (respectively, the cardinality of $\N\backslash S$) and we denote it by $F(S)$ (respectively, $g(S)$). These invariants have been widely studied (see \cite{alfonsin}) and they, together with the embedding dimension, are the background of one the most important problems in this theory: Wilf's conjecture which stablishes that
if $S$ is a numerical semigroup then $e(S)g(S)\leq(e(S)-1)(F(S)+1)$ (see \cite{wilf}).
Nowadays, it is still open.

In this work, we show that if we go along through the branches of the tree associated to $[S]$, the numerical semigroups have a greater Frobenius number and genus. These facts allow us to give an algorithm for building all the elements of $\mathcal{L}(m,e)$ with a fixed Frobenius number and/or genus.
Finally, in order to compute the Frobenius number and the genus of the numerical semigroups of $[S]$, we give an algorithm based in \cite{contejean}.

The content of this work is organized as follows.
In Section \ref{section_partition}, a partition of the set ${\mathcal L}(m,e)$ is studied and we establish an application $\phi:{\mathcal L}(m,e)\to {\mathcal C}(m,e)$ such that $[S]\cap {\mathcal C}(m,e)$ is equal to $\{\phi(S)\}$ for every $S\in{\mathcal L}(m,e)$.
Theorem \ref{Teorema9}, in Section \ref{S3}, allows us to compute recursively the elements of $[S]$.
In Section \ref{S4}, we give some algorithms for computing the elements of $[S]$ with Frobenius number and/or genus less than fixed integer numbers.
Finally, in Section \ref{S5}, we show how the Apery set of the elements of $[S]$ allows us to compute easily their Frobenius number and genus. Besides, we also check that Wilf's conjecture is satisfied for some elements of $[S]$.

\section{A partition of $\mathcal{L}(m,e)$}\label{section_partition}

If $A$ and $B$ are subsets of $\N$ we denote by $A+B=\{a+b\mid a\in A\mbox{\ and\ } b\in B\}$.

It is well known (for example see Proposition 2.10 from \cite{libro}) that if $S$ is a numerical semigroup then $e(S)\leq m(S)$. Note that if $e(S)=1$ then $S=\N$. Therefore, in the sequel, we assume that $e$ and $m$ are integers such that $2\leq e\leq m$.

Given $S\in\mathcal{L}(m,e)$ we denote by $\phi(S)$ the numerical semigroup generated by $\{m\}+\{x\mod m\mid x\in \msg(S)\}$. Clearly, $\phi(S)$ is a packed numerical semigroup and therefore we have the following result.

\begin{lemma}\label{Lema1}
With the previous assumptions, $\phi$ defines a surjective map from $\mathcal{L}(m,e)$ to $\mathcal{C}(m,e)$.
\end{lemma}

We define in $\mathcal{L}(m,e)$ the following equivalence relation: $S\mathcal{R}T$ if $\phi(S)=\phi(T)$. Given $S\in\mathcal{L}(m,e),$ $[S]$ denotes the set $\{T\in\mathcal{L}(m,e)\mid S\mathcal{R}T\}$. Therefore, the quotient set $\displaystyle{{\mathcal{L}(m,e)}/{\mathcal{R}}}=\{[S]\mid S\in\mathcal{L}(m,e)\}$ is a partition of $\mathcal{L}(m,e)$.

\begin{lemma}\label{Lema2}
If $S\in\mathcal{L}(m,e)$, then $[S]\cap\mathcal{C}(m,e)=\{\phi(S)\}$.
\end{lemma}

\begin{proof}
By Lemma \ref{Lema1}, we know that $\phi(S)\in\mathcal{C}(m,e)$. Moreover, it is clear that $\phi(\phi(S))=\phi(S)$. Therefore, $S\mathcal{R}\phi(S)$ and $\phi(S)\in[S]\cap\mathcal{C}(m,e)$.

If $T\in [S]\cap\mathcal{C}(m,e)$, then $\phi(T)=\phi(S)$ and $\phi(T)=T$, so $T=\phi(S)$.
\end{proof}

The following result is a consequence of the previous lemmas.

\begin{theorem}\label{Teorema3}
Let $m$ and $e$ be integers such that $2\leq e\leq m$. Then $\{[S]\mid S\in\mathcal{C}(m,e)\}$ is a partition of $\mathcal{L}(m,e)$. Moreover, if $\{S,T\}\subseteq\mathcal{C}(m,e)$ and $S\neq T$ then $[S]\cap[T]=\emptyset$.
\end{theorem}

Therefore, as a consequence of Theorem \ref{Teorema3}, for computing all the elements of the set $\mathcal{L}(m,e)$ is only necessary to do the following steps:
\begin{enumerate}
    \item Compute $\mathcal{C}(m,e)$.
    \item For every $S\in\mathcal{C}(m,e)$ compute $[S]$.
\end{enumerate}

$\mathcal{C}(m,e)$ is easy to compute using the following result.

\begin{proposition}\label{Proposition4}
Let $m$ and $e$ be integers such that $2\leq e\leq m,$ and let $A$ be a subset of $\{1,\ldots,m-1\}$ with cardinality $e-1$ such that $\gcd(A\cup\{m\})=1$. Then $S=\langle \{m\}+(A\cup\{0\}) \rangle\in\mathcal{C}(m,e)$. Moreover, every element of $\mathcal{C}(m,e)$ has this form.
\end{proposition}

\begin{proof}
The set $S$ is a numerical semigroup because  $\gcd(\{m\}+(A\cup\{0\}))=\gcd(A\cup\{m\})=1$. It is straightforward to prove that $\msg(S)=\{m\}+(A\cup\{0\}),$ so $S\in\mathcal{C}(m,e)$.

If $S\in\mathcal{C}(m,e)$ then $\msg(S)=\{m,m+r_1,\ldots,m+ r_{e-1}\}$ with $\{r_1,\ldots,r_{e-1}\}\subseteq\{1\ldots,m-1\}$. Moreover, since $\gcd\{m,m+r_1,\ldots,m+r_{e-1}\}=1$, $\gcd\{m,r_1,\ldots,r_{e-1}\}=1$.
\end{proof}

We illustrate the content of the previous proposition with an example.

\begin{example}\label{Ejemplo5}
We are going to compute the set $\mathcal{C}(6,3)$ formed by all the packed numerical  semigroups of multiplicity 6 and embedding dimension 3. For this purpose, and using Proposition \ref{Proposition4}, it is enough computing the subsets $A$ of $\{1,2,3,4,5\}$ of cardinality 2 such that $\gcd(A\cup\{6\})=1$. These sets are $\{1,2\}$, $\{1,3\}$, $\{1,4\}$, $\{1,5\}$, $\{2,3\}$, $\{2,5\}$, $\{3,4\}$, $\{3,5\}$ and $\{4,5\}$. Therefore, $\mathcal{C}(6,3)$ is $\{\langle 6,7,8 \rangle, \langle 6,7,9 \rangle, \langle 6,7,10 \rangle,\langle 6,7,11 \rangle, \langle 6,8,9 \rangle, \langle 6,8,11 \rangle, \langle 6,9,10 \rangle, \langle 6,9,11 \rangle, \langle 6,10,11 \rangle\}.$
\end{example}

Note that if $m$ is a prime number then every subset $A$ of $\{1,\ldots,m-1\}$ with cardinality $e-1$ verifies that $\gcd(A\cup\{m\})=1$. Therefore, we have the following result.

\begin{proposition}\label{Proposicion6}
If $m$ is a prime number and $e$ is an integer number such that $2\leq e\leq m$ then $\mathcal{C}(m,e)$ has cardinality $\binom{m-1}{e-1}$.
\end{proposition}

Our next goal in this work is to show a recursive procedure that allows us to compute $[S]$ for every $S\in\mathcal{C}(m,e)$. In order to achieve it, in the next section, we set the elements of $[S]$ in a tree.

\section{The tree associated to $[S]$}\label{S3}
A graph $G$ is pair $(V,E)$ where $V$ is a set (with elements called vertices) and $E$ is a subset of $\{(v,w)\in V\times V\mid v\neq w\}$ (with elements called edges). A path which connects the vertices $x$ and $y$ of $G$ is a sequence of different edges of the form $(v_0,v_1), (v_1,v_2),\ldots,(v_{n-1},v_n)$ such that $v_0=x$ and $v_n=y$.

A graph $G$ is a tree if there is a vertex $r$ (known as the root of $G$) such that for any other vertex $x$ of $G$ there exists a unique path connecting $x$ and $r$. If $(x,y)$ is an edge of a tree, we say that $x$ is a child of $y$.

\begin{lemma}\label{Lema7}
If $\{n_1<n_2<\cdots<n_e\}$ is a minimal system of generators of a numerical semigroup and $n_e-n_1>n_1$ then $\{n_1,\ldots,n_{e-1},n_e-n_1\}$ is also a minimal system of generators of a numerical semigroup.
\end{lemma}

\begin{proof}
In other case, there exists $k\in\{1,\ldots,e-1\}$ such that $n_k\in\{n_e-n_1\}+\langle n_1,\ldots,n_{k-1},n_{k+1},\dots, ,n_{e-1},n_e-n_1\rangle$. But it is not possible because $n_e-n_1+n_1=n_e>n_k$.
\end{proof}

Let $S$ be a numerical semigroup. We denote by $M(S)$ the maximum of $\msg(S).$
If $S\in\mathcal{L}(m,e),$ we define the following sequence of elements of $\mathcal{L}(m,e)$:
\begin{itemize}
	\item $S_0 = S$,
	\item $S_{n+1} = \langle (\msg(S_n)\backslash\{M(S_n)\})\cup\{M(S_n)-m\} \rangle$ if $M(S_n)-m>m$.
\end{itemize}
Because of Lemma \ref{Lema7}, there exists a sequence: $S=S_0\subsetneq S_1\subsetneq\ldots\subsetneq S_k=\phi(S)\in\mathcal{C}(m,e)$.

\begin{example}\label{Ejemplo8}
Let $S\in \mathcal{L}(5,3)$ be the semigroup minimally generated by $\{ 5,13,21 \}.$ Then, we have the following sequence of elements of $\mathcal{L}(5,3)$: $S_0=\langle 5,13,21\rangle\subsetneq S_1=\langle 5,13,16 \rangle\subsetneq S_2=\langle 5,11,13 \rangle\subsetneq S_3=\langle 5,8,11 \rangle\subsetneq S_4=\langle 5,6,8 \rangle=\phi(S)\in\mathcal{C}(5,3)$.
\end{example}

Let $S$ be in $\mathcal{C}(m,e)$. We define the graph $G([S])$ as follows: $[S]$ is the set of vertices and $(A,B)\in [S]\times[S]$ is an edge if $\msg(B)=(\msg(A)\backslash\{M(A)\})\cup\{M(A)-m\}$ with .

\begin{theorem}\label{Teorema9}
If $S\in\mathcal{C}(m,e)$ then $G([S])$ is a tree with root $S$. Moreover, if $P\in [S]$ and $\msg(P)=\{n_1<n_2<\cdots<n_e\}$ then the children of $P$ in $G([S])$ are the numerical semigroups of the form $\langle (\{n_1,\ldots,n_e\}\backslash\{n_k\})\cup\{n_k+n_1\}\rangle$ such that $k\in\{2,\ldots,e\}$, $n_k+n_1>n_e$ and $n_k+n_1\notin\langle \{n_1,\ldots,n_e\}\backslash\{n_k\} \rangle$.
\end{theorem}

\begin{proof}
From the definition and the comment after Lemma \ref{Lema7}, we have that $G([S])$ is a tree with root $S$.

Let $k$ be in $\{2,\ldots,e\}$ such that $n_k+n_1>n_e$ and $n_k+n_1\notin\langle \{n_1,\ldots,n_e\}\backslash\{n_k\} \rangle$. If $H=\langle (\{n_1,\ldots,n_e\}\backslash\{n_k\})\cup\{n_k+n_1\} \rangle$ is clear that $\msg(H)=(\{n_1,\ldots,n_e\}\backslash\{n_k\})\cup\{n_k+n_1\}$ and $\msg(P)=(\msg(H)\setminus\{M(H)\})\cup\{M(H)-m\}$. Therefore $H$ is a child of $P$.

Conversely, if $H$ is a child $P$ then $(H,P)$ is an edge of $G([S])$ and we obtain that $H$ is as the theorem describes.
\end{proof}

The previous theorem allows us to build recursively the elements of $[S]$ as it is shown in the next example.

\begin{example}\label{Ejemplo11}
Figure \ref{arbol5_6_8} shows some levels of the tree $G([\langle 5,6,8 \rangle])$.

\begin{figure}[!h]
%\hspace*{-1cm}
\scriptsize
\xymatrix@C=0.5em{
 & & & &                                       &                                        & \langle 5,6,8\rangle \ar@{<-}[rd]\ar@{<-}[ld]  &                       \\
 & & & &                                       & \langle 5,8,11\rangle \ar@{<-}[d]           &                                      & \langle 5,6,13\rangle \\
 & & & &                                       & \langle 5,11,13\rangle \ar@{<-}[rd]\ar@{<-}[ld]  &                                      &                       \\
 & & & & \langle 5,13,16\rangle \ar@{<-}[rd]\ar@{<-}[ld] &                                        &      \langle 5,11,18\rangle \ar@{<-}[d]  &                       \\
 & & & \langle 5,16,18\rangle \ar@{<-}[rd]\ar@{<-}[ld] &  &    \langle 5,13,21 \rangle       &      \langle 5,11,23\rangle \ar@{<-}[d]  &                       \\
 & & \langle 5,18,21\rangle \ar@{<-}[rd]\ar@{<-}[ld] &  & \langle 5,16,23\rangle \ar@{<-}[rd] &     &      \langle 5,11,28\rangle  &                       \\
 & \langle 5,21,23\rangle \ar@{<-}[rd]\ar@{<-}[ld] &  & \langle 5,18,26\rangle \ar@{<-}[rd] &  &   \langle 5,16,28\rangle\ar@{<-}[rd]  &        &                       \\
 \langle 5,23,26\rangle &   & \langle 5,21,28\rangle &  & \langle 5,18,31\rangle &     & \langle 5,16,33\rangle       &                       \\
}\caption{Seven levels of the tree of the packed numerical  semigroup $\langle 5,6,8 \rangle.$}\label{arbol5_6_8}
\end{figure}
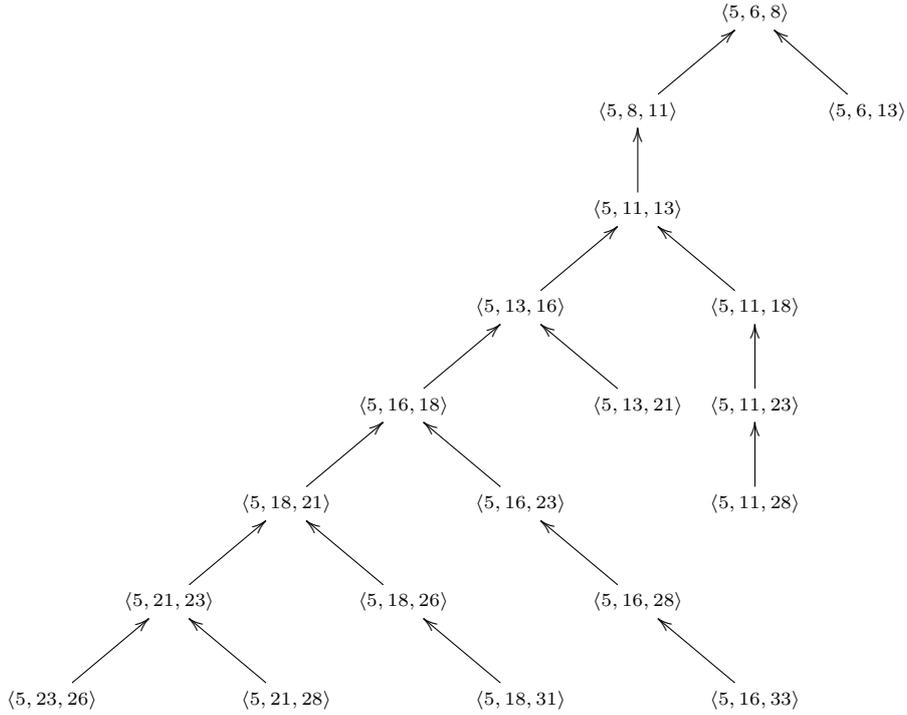

\end{example}

Note that the cardinality  of $[S]$ is infinity, so it is impossible to compute all the elements of $[S]$. However, in the next section, we show that it is possible to  compute all the elements of $[S]$ with a fixed Frobenius number or genus.

\section{Frobenius number and genus}\label{S4}

Let $P$ be a numerical semigroup with minimal generating set $\{n_1<n_2<\cdots<n_e\}$, $k\in\{2,\ldots,e\}$ and $H$ be the numerical semigroup generated by $(\{n_1,\ldots,n_e\}\backslash\{n_k\})\cup\{n_k+n_1\}$. Then $H\subset P$, $F(P)\leq F(H)$ and $g(P)<g(H)$. We can formulate the following result.

\begin{proposition}\label{Proposicion12}
If $S\in\mathcal{C}(m,e)$, $P\in[S]$ and $(H,P)$ is an edge of $G([S])$ then $F(P)\leq F(H)$ and $g(P)<g(H)$.
\end{proposition}

From previous proposition, for every semigroup $S$
the numerical semigroups obtained from it have a greater or equal Frobenius number and a greater genus than $S$. These facts allow us to formulate Algorithm \ref{Algoritmo13} and \ref{Algoritmo16} for computing all the elements in $[S]$ with Frobenius number less than or equal to a given integer and genus less than or equal to another given integer.

\begin{algorithm}[H]
\caption{Sketch of the algorithm to determinate the elements of $T\in [S]$ such that $F(T)\leq F$ for a fixed integer $F$.}\label{Algoritmo13}
\textbf{INPUT:} $(S,F)$ where $S$ is a packed numerical  semigroup and $F$ is a positive integer.\\
\textbf{OUTPUT:}  $\{T\in[S]\mid F(T)\leq F\}$.
\begin{algorithmic}[1]
    \If{$F(S)>F$}
        \State \Return{$\emptyset$}
    \EndIf
    \While{true}
	   \State $A=\{S\}$ and $B=\{S\}$.
	   \State $C=\left\{H \mid
			   %H \mbox{ is a numerical semigroup},
			   H \mbox{ is a child of an element of }B,~ F(H)\leq F\right\}$.
    	\If{$C=\emptyset$}
            \State \Return{$A$}
        \EndIf
	   \State $A=A\cup C$, $B=C$.
    \EndWhile
\end{algorithmic}
\end{algorithm}

Next example illustrates how the previous algorithm works.

\begin{example}\label{Ejemplo14}
We compute all the elements of $[\langle 5,6,8 \rangle]$ with Frobenius number less than or equal to  25.
\begin{itemize}
	\item $A=\{\langle 5,6,8 \rangle\}$, $B=\{\langle 5,6,8 \rangle\}$ and $C=\{\langle 5,8,11 \rangle, \langle 5,6,13 \rangle \}$.
	\item $A=\{\langle 5,6,8 \rangle, \langle 5,8,11 \rangle, \langle 5,6,13 \rangle \}$, $B=\{\langle 5,8,11 \rangle, \langle 5,6,13 \rangle\}$ and $C=\{\langle 5,11,13 \rangle \}$.
	\item $A=\{\langle 5,6,8 \rangle, \langle 5,8,11 \rangle, \langle 5,6,13 \rangle, \langle 5,11,13 \rangle \}$, $B=\{\langle 5,11,13 \rangle\}$ and $C=\{\langle 5,11,18 \rangle \}$.
	\item $A=\{\langle 5,6,8 \rangle, \langle 5,8,11 \rangle, \langle 5,6,13 \rangle, \langle 5,11,13 \rangle, \langle 5,11,18 \rangle \}$, $B=\{\langle 5,11,18 \rangle\}$ and $C=\emptyset$.
\end{itemize}
Therefore, the set $\{T\in[\langle 5,6,8 \rangle]\mid F(T)\leq 25\}$ is equal to $\{\langle 5,6,8 \rangle, \langle 5,8,11 \rangle, \langle 5,6,13 \rangle, \langle 5,11,13 \rangle, \langle 5,11,18 \rangle \}$.
\end{example}

Next algorithm allows us to compute all the numerical semigroups with multiplicity $m$, embedding dimension $e$ and Frobenius number less than or equal to $F$. Note that if $S$ is a numerical semigroup, such that $S\neq\N$ then $m(S)-1\notin S$ and then $m(S)-1\leq F(S)$.

\begin{algorithm}[H]
\caption{Sketch of the algorithm to determinate the numerical semigroups with a fixed embedding dimension and bounded Frobenius number.}\label{Algoritmo15}
\textbf{INPUT:} $m,e,$ and $F$ positive integers such that $2\leq e\leq m\leq F+1$.\\
\textbf{OUTPUT:} $\{S\mid S\mbox{ numerical semigroup, } m(S)=m,\ e(S)=e\mbox{ and }F(S)\leq F\}$.
\begin{algorithmic}[1]
	\State compute $\mathcal{C}(m,e)$, using Proposition \ref{Proposition4}.
    \ForAll{$S\in\mathcal{C}(m,e)$}
        \State compute $A(S)=\{T\in[S]\mid F(T)\leq F\}$, using Algorithm \ref{Algoritmo13}.
    \EndFor
    \State \Return{$\cup_{S\in\mathcal{C}(m,e)}A(S)$}
\end{algorithmic}
\end{algorithm}

Now, we exchange the concept for Frobenius number for the genus in Algorithm \ref{Algoritmo13} and \ref{Algoritmo15}.

\begin{algorithm}[H]
\caption{Sketch of the algorithm to determinate the numerical semigroups with bounded genus.}\label{Algoritmo16}
\textbf{INPUT:} $(S,g)$ where $S$ is a packed numerical  semigroup and $g$ is a positive integer.\\
\textbf{OUTPUT:}  $\{T\in[S]\mid g(T)\leq g\}$.
\begin{algorithmic}[1]
	\If{$g(S)>g$}
        \State \Return{$\emptyset$}
    \EndIf
	\State $A=\{S\}$ and $B=\{S\}$.
    \While{true}
        \State $C=\{H \mid
        ~ H \mbox{ is a child of an element of }B,~
         g(H)\leq g\}$.
        \If{$C=\emptyset$}
            \State \Return{$A$}
        \EndIf
        \State $A=A\cup C$, $B=C$
    \EndWhile
\end{algorithmic}
\end{algorithm}

Next example illustrates how the previous algorithm works.

\begin{example}\label{Ejemplo17}
We compute all the elements of $[\langle 5,6,8 \rangle]$ with genus less than or equal to 15.
\begin{itemize}
	\item $A=\{\langle 5,6,8 \rangle\}$, $B=\{\langle 5,6,8 \rangle\}$ and $C=\{\langle 5,8,11 \rangle, \langle 5,6,13 \rangle \}$.
	\item $A=\{\langle 5,6,8 \rangle, \langle 5,8,11 \rangle, \langle 5,6,13 \rangle \}$, $B=\{\langle 5,8,11 \rangle, \langle 5,6,13 \rangle\}$ and $C=\{\langle 5,11,13 \rangle \}$.
	\item $A=\{\langle 5,6,8 \rangle, \langle 5,8,11 \rangle, \langle 5,6,13 \rangle, \langle 5,11,13 \rangle \}$, $B=\{\langle 5,11,13 \rangle\}$ and $C=\{\langle 5,11,18 \rangle \}$.
	\item $A=\{\langle 5,6,8 \rangle, \langle 5,8,11 \rangle, \langle 5,6,13 \rangle, \langle 5,11,13 \rangle, \langle 5,11,18 \rangle \}$, $B=\{\langle 5,11,18 \rangle\}$ and $C=\emptyset$.
\end{itemize}
Algorithm \ref{Algoritmo16} returns $\{\langle 5,6,8 \rangle, \langle 5,8,11 \rangle, \langle 5,6,13 \rangle, \langle 5,11,13 \rangle, \langle 5,11,18 \rangle \}$.
\end{example}

Note that if $S$ is a numerical semigroup such that $S\neq\N$ then $\{1,\ldots,m(S)-1\}\subseteq\N\backslash S$ and then $m(S)-1\leq g(S)$.

Next algorithm is obtained combining the previous results.

\begin{algorithm}[H]
\caption{Sketch of an algorithm for computing numerical semigroups with fixed multiplicity, embedding dimension and bounded genus}\label{Algoritmo18}
\textbf{INPUT:} $m,e,$ and $g$ positive integers such that $2\leq e\leq m\leq g+1$.\\
\textbf{OUTPUT:} $\{S\mid S\mbox{ numerical semigroup, } m(S)=m,\ e(S)=e\mbox{ and }g(S)\leq g\}$.
\begin{algorithmic}[1]
	\State compute $\mathcal{C}(m,e)$, using Proposition \ref{Proposition4}.
    \ForAll{$S\in\mathcal{C}(m,e)$}
        \State compute $A(S)=\{T\in[S]\mid g(T)\leq g\}$, using Algorithm \ref{Algoritmo16}.
    \EndFor
    \State \Return{$\cup_{S\in\mathcal{C}(m,e)}A(S)$}
\end{algorithmic}
\end{algorithm}

Note that applying Algorithm \ref{Algoritmo13} and \ref{Algoritmo15} we have to compute the Frobenius number and the genus, respectively, of the numerical semigroups we obtain recursively when we build $[S]$. Results of the next section allow us to compute easily the Frobenius number and the genus of every semigroup of $[S]$.

\section{The Apery set of the elements of $[S]$}\label{S5}

Let $S$ be a numerical semigroup and $n\in S\backslash\{0\}$. The Apery set (named by \cite{apery}) of $n$ in $S$ is $\Ap(S,n)=\{s\in S\mid s-n\notin S\}$. Next result is a consequence of Lema 2.4 from \cite{libro}.

\begin{lemma}\label{Lema19}
Let $S$ be a numerical semigroup and $n\in S\backslash\{0\}$. Then $\Ap(S,n)$ has cardinality $n$. Moreover, $\Ap(S,n)=\{w(0)=0,w(1),\ldots,w(n-1)\}$ where $w(i)$ is the less element in $S$ congruent with $i$ modulo $n$.
\end{lemma}

The set $\Ap(S,n)$ give us a lot of information of $S$.
The following result is found in $\cite{selmer}$.

\begin{lemma}\label{Lema20}
Let $S$ be a numerical semigroup and $n\in S\backslash\{0\}$. Then:
\begin{itemize}
    \item $F(S)=\max(\Ap(S,n))-n$.
    \item $g(S)=\frac{1}{n}(\sum_{w\in \Ap(S,n)}w)-\frac{n-1}{2}$.
\end{itemize}
\end{lemma}

The following result is a consequence of Lemma \ref{Lema19}.

\begin{lemma}\label{Lema21}
Let $S$ be a numerical semigroup with minimal system of generators $\{n_1,n_2,\ldots,n_e\}$ and $\Ap(S,n_1)=\{0,w(1),\ldots,w(n_1-1)\}$. Then $w(i)=\min\{a_2n_2+\cdots+a_en_e\mid (a_2,\ldots,a_e)\in\N^{e-1}\mbox{ and } a_2n_2+\cdots+a_en_e\equiv i \mod n_1\}$.
\end{lemma}

Note that the set $\{(a_2,\ldots,a_e)\in\N^{e-1} \mid a_2n_2+\cdots+a_en_e\equiv i\mod n_1\}$ has a finite number of minimal elements (using the usual ordering in $\N^{e-1}$) by Dickson's Lemma (Theorem 5.1 from \cite{finitely}). We denote the set of these minimal elements by $\mathcal{M}((n_1,\ldots,n_e),i)$. Next result is a obtained from Lemma \ref{Lema21}.

\begin{proposition}\label{Proposicion22}
Let $S$ be a numerical semigroup with minimal system of generators $\{n_1,n_2,\ldots,n_e\}$ and $\Ap(S,n_1)=\{0,w(1),\ldots,w(n_1-1)\}$. Then $w(i)=\min\{a_2n_2+\cdots+a_en_e\mid (a_2,\ldots,a_e)\in\mathcal{M}((n_1,\ldots,n_e),i)\}$.
\end{proposition}

Next example illustrates the previous results.

\begin{example}\label{Ejemplo23}
In this example we try to compute the Apery set of the numerical semigroups of $[\langle 5,6,8 \rangle]$ that we obtained in Example \ref{Ejemplo11}.

For every $i\in\{1,2,3,4\}$ let $A(i)$ be the set $\{(a_2,a_3)\in\N^2\mid a_2\cdot 1+a_3\cdot 3\equiv i \mod 5\},$ and let $\mathcal{M}(i)$ be the set of the minimal elements of $A(i)$. Then, $\mathcal{M}(1)=\{(1,0),(0,2)\}$, $\mathcal{M}(2)=\{(2,0),(0,4),(1,2)\}$, $\mathcal{M}(3)=\{(3,0),(0,1)\}$ and $\mathcal{M}(4)=\{(4,0),(0,3),(1,1)\}$.

Now, if we take an element from $[\langle 5,6,8 \rangle]$, for example $S=\langle 5,21,13 \rangle,$ and we want to compute $\Ap(S,5)=\{0,w(1),w(2),w(3),w(4)\}$, by applying Proposition \ref{Proposicion22} we have that $w(1)=\min\{21,26\}=21$, $w(2)=\min\{42,52,47\}=42$, $w(3)=\min\{63,13\}=13$ and $w(4)=\min\{84,39,34\}=34$.
\end{example}

Note that in the previous example it was easy to compute $\mathcal{M}(i)$ for every $i\in\{1,2,3,4\}$. Now, we want to give an algorithm that always allows us to compute $\mathcal{M}((n_1,\ldots,n_e),i)$. In order to do it, we introduce the following sets:
$$C(1)=\{(x_2,\ldots,x_e)\in\N^{e-1}\mid n_2x_2+\cdots+n_ex_e\equiv i\mod n_1\},$$
$$C(2)=\{(x_1,x_2,\ldots,x_e)\in\N^e\mid (-n_1)x_1+n_2x_2+\cdots+n_ex_e = i\},$$
$$C(3)=\{(x_1,x_2,\ldots,x_e,x_{e+1})\in\N^{e+1}\mid (-n_1)x_1+n_2x_2+\cdots+n_ex_e+(-i)x_{e+1} = 0\}.$$

\begin{lemma}\label{Lema24}
If $(a_2,\ldots,a_e)\in C(1)$ then there exists $a_1\in\N$ such that $(a_1,a_2,\ldots,a_e)\in C(2)$.
\end{lemma}

\begin{proof}
It is enough to note that if $n_2a_2+\cdots+n_ea_e\equiv i\mod n_1$ then, there exist $a_1\in\N$ such that $n_2a_2+\cdots+n_ea_e = i+a_1n_1$.
\end{proof}

Thanks to \cite{JLMS} we know that $C(3)$ is a finitely generated submonoid of $\N^{e+1}$. Next result can be deduced from Lemma 2 of $\cite{JLMS}$.

\begin{lemma}\label{Lema25}
Let $A$ be the set $\{\alpha_1,\ldots,\alpha_t\}$ with $\alpha_i=(\alpha_{i1},\alpha_{i2},\ldots,\alpha_{ie},\alpha_{i\,e+1})$ a system of generators of $C(3)$. If we suppose that $\alpha_1,\ldots,\alpha_d$ are the elements in $A$ with the last coordinate equal to zero and $\alpha_{d+1},\ldots,\alpha_q$ are the elements of $S$ with the last coordinate equal to 1, then $C(2)=\{\bar{\alpha}_{d+1},\ldots\bar{\alpha}_q\}+\langle \bar{\alpha}_1,\ldots,\bar{\alpha}_d \rangle$ where $\bar{\alpha}_i=(\alpha_{i1},\alpha_{i2},\ldots,\alpha_{ie})$.
\end{lemma}

Note that $\mathcal{M}((n_1,\ldots,n_e),i)$ are the minimal elements of $C(1)$. Hence, the following result allows us to compute it.

\begin{proposition}\label{Proposicion26}
The minimal elements of $C(1)$ are the same that the minimal elements of the set $\{(\alpha_{d+1\,2},\ldots,\alpha_{d+1\,e}),\ldots,(\alpha_{q2},\ldots,\alpha_{qe})\}$.
\end{proposition}

\begin{proof}
Let $k$ be in $\{d+1,\ldots,q\}$. We check if $(\alpha_{k2},\ldots,\alpha_{ke})\in C(1)$. Since $(\alpha_{k1},\ldots,\alpha_{ke},1)\in C(3)$, then $(-n_1)\alpha_{k1}+n_2\alpha_{k2}+\cdots+n_e\alpha_{ke}-i=0$. Therefore $n_2\alpha_{k2}+\cdots+n_e\alpha_{ke}\equiv i \mod n_1$ so $(\alpha_{k2},\ldots,\alpha_{ke})\in C(1)$.

We finish the proof checking that if $(a_2,\ldots,a_e)\in C(1)$ then there exists $k\in\{d+1,\ldots,q\}$ such that $(\alpha_{k2},\ldots,\alpha_{ke})\leq(a_2,\ldots,a_e)$. By Lemma \ref{Lema24}, there exists $a_1\in\N$ such that $(a_1,a_2,\ldots,a_e)\in C(2)$. Hence by Lemma \ref{Lema25}, there exists $k\in\{d+1,\ldots,q\}$ such that $(\alpha_{k1},\alpha_{k2},\ldots,\alpha_{ke})\leq(a_1,a_2,\ldots,a_e)$. Therefore, we have that $(\alpha_{k2},\ldots,\alpha_{ke})\leq(a_2,\ldots,a_e)$.
\end{proof}

A efficient algorithm for computing a finite system of generators of $C(3)$ is given in \cite{contejean}. So, applying the previous result we have an algorithm which allows us to compute the minimal elements of $C(1)$. Therefore, using Proposition \ref{Proposicion22} and the idea exposed in Example \ref{Ejemplo23}, we have an algorithm for computing easily $\Ap(T,m)$ for every $T\in[S]$. Finally, thanks to Lemma \ref{Lema20} we can compute easily $F(T)$ and $g(T)$ for every $T\in [S]$.

\section{Examples}
We devote this section to
illustrate the previous results with several examples. They show all the semigroups with a fixed multiplicity, embedding dimension, and Frobenius number or genus. Besides, we check Wilf's conjecture for many semigroups in the tree associated to $[S]$ for several packed numerical  semigroups. The computations have been done in an Intel i7 with 32 Gb of RAM, and using {\tt Mathematica}  (\cite{mathematica}).

\begin{example}\label{example_m_e_F}
In this example we compute all the numerical semigroups with multiplicity 6, embedding dimension 3, and Frobenius number equal to 23.

With these fixed conditions, the set $\mathcal{C}(m,e)$ is
\begin{multline*}
\{\langle6, 7, 8\rangle, \langle6, 7, 9\rangle, \langle6, 7, 10\rangle, \langle6, 7, 11\rangle, \langle6, 8, 9\rangle,\\ \langle6, 8,11\rangle, \langle6, 9, 10\rangle, \langle6, 9, 11\rangle, \langle6, 10, 11\rangle\}.
\end{multline*}
The Frobenius number of these semigroups are 17, 17, 15, 16, 19, 21, 23, 25 and 25, respectively. So, by Proposition \ref{Proposicion12}, for computing the semigroups with Frobenius number 23, we only consider the packed numerical  semigroups $L=\{\langle6, 7, 8\rangle, \langle6, 7, 9\rangle, \langle6, 7, 10\rangle, \langle6, 7, 11\rangle, \langle6, 8, 9\rangle, \langle6, 8,
11\rangle, \langle6, 9, 10\rangle\}.$ Applying Algortihm \ref{Algoritmo15}, we compute the elements in $G([S])$ with the fixed Frobenius number. For example, from the first packed numerical  semigroups in $L$ only one numerical semigroup with Frobenius number equal to 23 is obtained (see Figure \ref{arbol6_7_8_Frob}),
but there is no numerical semigroups with Frobenius number equal to 23 in $G[\langle6, 8, 9\rangle]$ (see Figure \ref{arbol6_8_9_Frob}).
Hence, the set of  numerical semigroups with multiplicity 6, embedding dimension 3, and Frobenius number equal to 23 is
$$\{ \langle6, 8, 13\rangle, \langle6, 7, 15\rangle, \langle6, 7, 22\rangle, \langle6, 7, 29\rangle,\langle6, 9, 10\rangle\}.$$

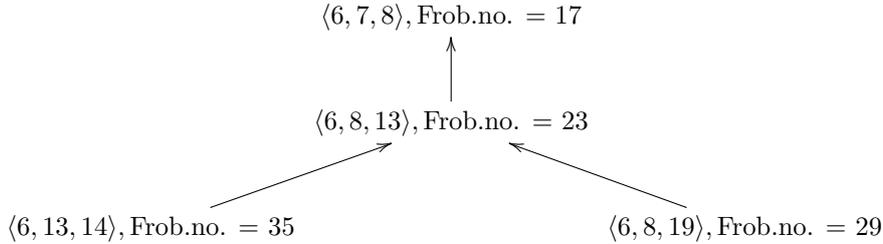
\begin{figure}[H]
\xymatrix@C=0.1em{
 & \langle 6, 7, 8\rangle, \text{Frob.no. = 17} \ar@{<-}[d]  &  \\
  & \langle 6,8,13\rangle, \text{Frob.no. = 23} \ar@{<-}[rd]\ar@{<-}[ld]  & \\
\langle 6,13,14 \rangle, \text{Frob.no. = 35} &  &  \langle 6,8,19\rangle, \text{Frob.no. = 29}
}\caption{Two levels of the tree associated to the semigroup $\langle 6, 7, 8\rangle.$}\label{arbol6_7_8_Frob}
\end{figure}

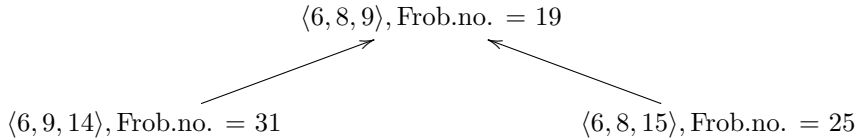
\begin{figure}[H]
\xymatrix@C=0.1em{
 & \langle 6,8,9\rangle, \text{Frob.no. = 19} \ar@{<-}[rd]\ar@{<-}[ld]  &  \\
\langle 6,9,14 \rangle, \text{Frob.no. = 31} &  &  \langle 6,8,15\rangle, \text{Frob.no. = 25}
}\caption{One level of the tree associated to the semigroup $\langle 6, 8, 9\rangle.$}\label{arbol6_8_9_Frob}
\end{figure}
\end{example}

\begin{example}
In this example, all the numerical semigroups with multiplicity 6, embedding dimension 3, and genus equal to 16 are computed. From Example \ref{example_m_e_F}, the set $\mathcal{C}(6,3)$ is
\begin{multline*}
\{\langle6, 7, 8\rangle, \langle6, 7, 9\rangle, \langle6, 7, 10\rangle, \langle6, 7, 11\rangle,\\ \langle6, 8, 9\rangle, \langle6, 8,11\rangle, \langle6, 9, 10\rangle, \langle6, 9, 11\rangle, \langle6, 10, 11\rangle\}.
\end{multline*}
The genus of these semigroups are 9, 9, 9, 10, 10, 11, 12, 13 and 13, respectively. So, by Proposition \ref{Proposicion12}, for computing the semigroups with genus 16, we apply Algorithm \ref{Algoritmo16} to all elements in $\mathcal{C}(6,3).$ For example, for the semigroups $\langle6, 7, 8\rangle$ and $\langle6, 8, 9\rangle$ we obtain the trees showed in Figures \ref{arbol6_7_8_genus} and \ref{arbol6_8_9_genus}, respectively.
Thus, the set of  numerical semigroups with multiplicity 6, embedding dimension 3, and genus 16 is
$$\{ \langle6, 14, 9\rangle, \langle6, 8, 21\rangle, \langle6, 15, 11\rangle, \langle6, 10, 17\rangle\}.$$

\begin{figure}[H]
\xymatrix@C=0.1em{
 & \langle 6, 7, 8\rangle, \text{genus = 9} \ar@{<-}[d]  &  \\
  & \langle 6,8,13\rangle, \text{genus = 12} \ar@{<-}[rd]\ar@{<-}[ld]  & \\
\langle 6,13,14 \rangle, \text{genus = 18} &  &  \langle 6,8,19\rangle, \text{genus = 15} \ar@{<-}[d] \\
 & &  \langle 6,8,25 \rangle, \text{genus = 18}
}
\caption{Three levels of the tree associated to the semigroup $\langle 6, 7, 8\rangle.$}\label{arbol6_7_8_genus}
\end{figure}
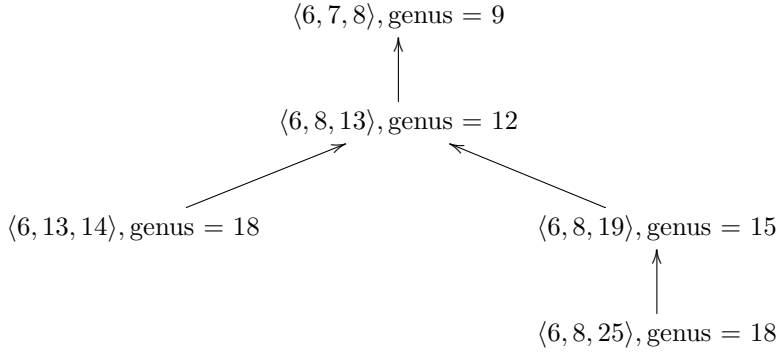
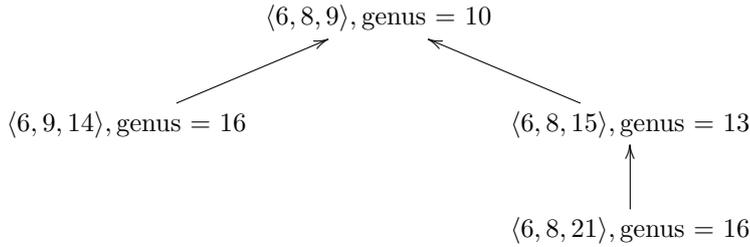
\begin{figure}[H]
\xymatrix@C=0.1em{
  & \langle 6,8,9\rangle, \text{genus = 10} \ar@{<-}[rd]\ar@{<-}[ld]  & \\
\langle 6,9,14 \rangle, \text{genus = 16} &  &  \langle 6,8,15\rangle, \text{genus = 13} \ar@{<-}[d] \\
 & &  \langle 6,8,21 \rangle, \text{genus = 16}
}
\caption{Two levels of the tree associated to the semigroup $\langle 6, 8, 9\rangle.$}\label{arbol6_8_9_genus}
\end{figure}

\end{example}

\begin{example}
Now, we check Wilf's conjecture for several elements in the tree associated to some packed numerical  semigroups. In this example, the elements are showed as a set with three entries ${{A},f,g}$ where $A$ is the minimal generating set of a numerical semigroup, and $f$ and $g$ are their Frobenius number and genus, respectively. Figure \ref{arbolwilf} illustrates two levels of the tree associated to the semigroup $S=\langle 110, 216, 217, 218, 219\rangle.$
Note that
for all its elements the inequality $\frac{e(S)}{e(S)-1}=\frac{5}{4}\leq \frac{F(S)+1}{g(S)}$ is held, and therefore they all satisfy Wilf's conjecture.

\begin{landscape}
\begin{figure}[!p]
\includegraphics[scale=.49]{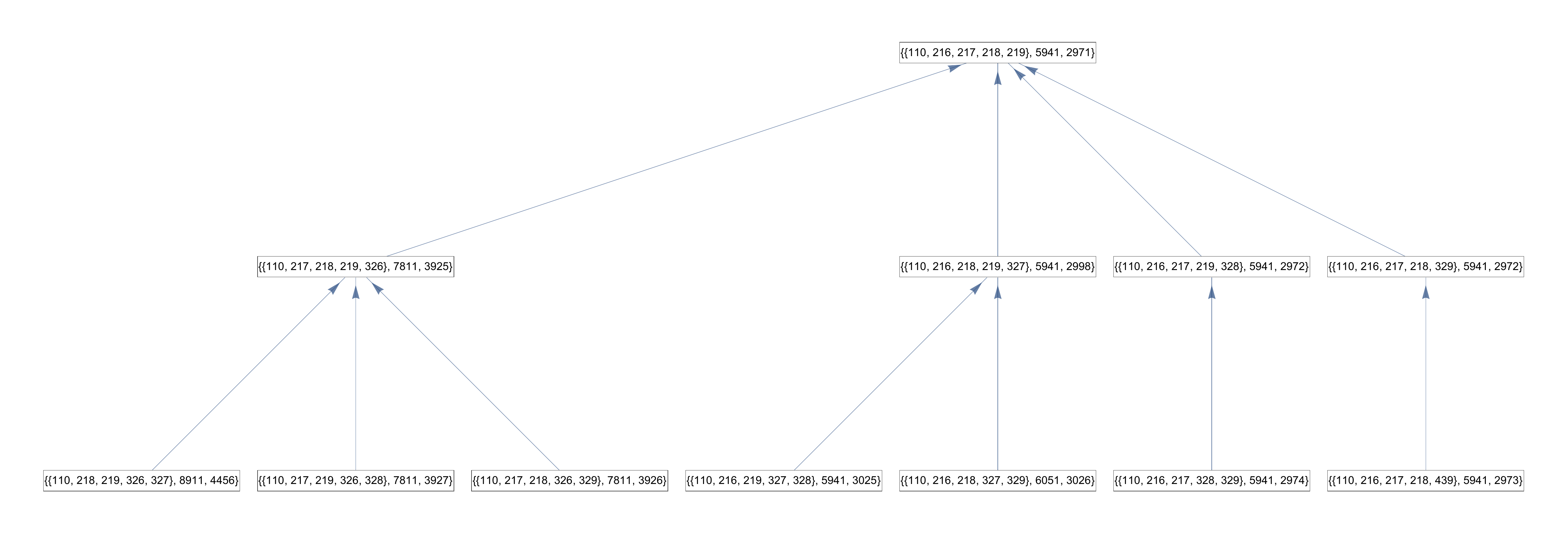}
\caption{Tree for checking Wilf's conjecture.}\label{arbolwilf}
\end{figure}
\end{landscape}

In Table \ref{tablawilf} we show some packed numerical  semigroups and the minimum and maximum of the quotients $(F(T)+1)/g(T)$ of the semigroups $T$ in their associated trees until a fixed level. Note that all tested semigroups (more than 66000) satisfy Wilf's conjecture.

\begin{table}[h]
\scriptsize
\centering
\begin{tabular}{|c|c|c|c|}
\hline
%\addlinespace
Semigroup & number & $\min \{ \frac{F(\bullet)+1}{g(\bullet)} \}$  & $\max \{ \frac{F(\bullet)+1}{g(\bullet)} \}$ \\
%& level & elements & & \\
%\addlinespace
\hline
%\addlinespace
$\{\{ 97, 111, 142, 159, 171 \} ,958 ,525 \}$ &  3694 & ${{1496}/{981}}$ & ${{2705}/{1357} }$ \\
%\addlinespace
\hline
%\addlinespace
$\{\{ 110, 216, 217, 218, 219\}, 5941, 2971\}$ &  425 & ${{2055}/{1081}}$ & 2 \\
%\addlinespace
\hline
%\addlinespace
$\{\{ 115, 151, 172, 189, 201 \},1282 , 724\}$ &  2656 & ${{1937}/{1224} }$ & ${{670}/{339} }$ \\
%\addlinespace
\hline
%\addlinespace
$\{\{ 111, 115, 122, 171, 181, 200, 201 \}, 702 , 445 \}$ &  35735 & ${{1488}/{1027}}$ & ${{2012}/{1041}}$ \\
%\addlinespace
\hline
%\addlinespace
$\{\{ 117, 125, 142, 173, 191, 203, 213 \}, 794, 476\}$ &  28688 & ${{382}/{261}}$ & ${{899}/{458}}$ \\
%\addlinespace
\hline
\end{tabular}
\caption{Checking Wilf's conjecture (up to level $15$).}\label{tablawilf}
\end{table}

\end{example}

Packed semigroups allow us to propose new problems for numerical semigroups. For example,
minimal values of $F(S)$ and $g(S)$ for all $S\in \mathcal{L}(m,e)$ can be studied by using this kind of semigroups. It is easy to prove that these minima are reached in elements belonging to $\mathcal{C}(m,e).$ So, $\min \{F(S)\mid S\in \mathcal{L}(m,e)\}$ and $\min \{g(S)\mid S\in \mathcal{L}(m,e)\}$ can be computed from the finite sets $\{F(S)\mid S\in \mathcal{C}(m,e)\}$ and $\{g(S)\mid S\in \mathcal{C}(m,e)\}$, respectively. Another interesting problem is to compute $\max \{F(S)\mid S\in \mathcal{C}(m,e)\}.$

\end{document}